\documentclass[11pt,a4paper]{article}
\usepackage[T1]{fontenc}
\usepackage[utf8]{inputenc}
\usepackage{lmodern}
\usepackage[margin=27mm]{geometry}
\usepackage{amsmath,amssymb,amsthm}
\usepackage{booktabs}
\usepackage{microtype}
\usepackage[hidelinks]{hyperref}
\hypersetup{pdftitle={A proof of the irreducibility conjecture for Legendre polynomials},pdfauthor={Zikang Deng}}
\newtheorem{theorem}{Theorem}[section]
\newtheorem{lemma}[theorem]{Lemma}
\newtheorem{proposition}[theorem]{Proposition}
\numberwithin{equation}{section}
\allowdisplaybreaks[1]
\begin{document}
\title{A proof of the irreducibility conjecture for Legendre polynomials}
\author{Zikang Deng\thanks{Correspondence: \href{mailto:202521130053@mail.bnu.edu.cn}{\texttt{202521130053@mail.bnu.edu.cn}}.}\\Beijing Normal University, Beijing, China}
\date{}
\maketitle

\begin{abstract}
We prove the irreducibility conjecture for Legendre polynomials proposed by Stieltjes: for every integer $j\geq 1$, both $P_{2j}(x)$ and $P_{2j+1}(x)/x$ are irreducible over the rational numbers. The proof proceeds by contradiction. Starting from a hypothetical nontrivial factorization, we remove the zero root from $P_n$ when present, multiply by a suitable constant, and express the resulting polynomial as a product of two integral polynomials $A(x^2)$ and $B(x^2)$. We then construct the resultant and its odd part,
\[
R=\operatorname{Res}_t(A,B),\qquad
\mathcal{R}=\frac{|R|}{2^{v_2(R)}}.
\]
Put $k=\deg A\leq\deg B$ and $m=\lfloor n/2\rfloor$. Using the classical auxiliary polynomial $U_n$ and the differential identity
\[
(1-x^2)(P_n'U_n-P_nU_n')=1-P_n^2,
\]
together with orthogonality, divisibility properties of the coefficients, and a least-common-multiple estimate, we obtain an upper bound for $\mathcal{R}$. On the other hand, the Legendre differential equation gives a lower bound for the absolute value of the derivative of $AB$ at each root of $A$. We use Chebyshev polynomials and Hadamard's inequality to bound the product of the squared pairwise differences of these roots, and combine this with estimates for the leading coefficients and the power of $2$ in $R$ to obtain a lower bound for $\mathcal{R}$. These estimates yield
\[
\begin{gathered}
m\log4-\log\frac{4(m+1)^2}{\sqrt m}
<\frac{\log\mathcal R}{k}
<\bigl(m+\sqrt{2n-1}\bigr)\log3
\qquad(n\ge64).
\end{gathered}
\]
For every $n\ge256$, the lower bound strictly exceeds the upper bound, giving a contradiction. Combining this with established irreducibility results and explicit integer comparisons for the remaining degrees proves irreducibility in every degree.
\end{abstract}

\noindent\textbf{Keywords:} Legendre polynomials; irreducibility; resultants; Newton polygons; $p$-adic valuations.

\noindent\textbf{Mathematics Subject Classification (2020):} Primary 11R09; Secondary 33C45.

\section{Introduction}

With the normalization $P_n(1)=1$, the Legendre polynomials are defined by Rodrigues' formula
\begin{equation}
P_n(x)=\frac{1}{2^n n!}\frac{d^n}{dx^n}(x^2-1)^n,
\qquad n\geq0.
\label{eq:1.1}
\end{equation}
They satisfy $P_n(-x)=(-1)^nP_n(x)$, so that $x$ is necessarily a factor when $n$ is odd. In his letter to Hermite dated October 2, 1890, Stieltjes asked whether Legendre polynomials are always irreducible over the rational numbers after this inevitable factor is removed \cite{stieltjes,cullinan}. We prove the following result.

\begin{theorem}[Stieltjes' irreducibility conjecture]
For every integer $j\geq1$, the polynomials $P_{2j}(x)$ and $P_{2j+1}(x)/x$ are irreducible in $\mathbb{Q}[x]$.
\end{theorem}

Wahab \cite{wahab} used Newton polygons to prove irreducibility for certain degrees; Cullinan and Hajir \cite{cullinan} further studied related irreducibility criteria and Galois groups. We give a uniform proof for all $n\geq256$ and complete the remaining degrees using the results collected in Theorem~1.8 and explicit integer comparisons.

The key arguments are two local bounds for the valuations of a resultant and their combination with a lower bound for the resultant obtained at real roots. The auxiliary polynomial and its differential identity come from the classical theory of Legendre functions; their explicit formulas and sources are given in Lemma~3.2.

\subsection{Preliminary theorems}

The notation needed in this subsection is specified in each theorem. Any further notation is defined in the section where it is used.

\begin{theorem}[Gauss' lemma and multiplicativity of the Gauss norm]
An integral polynomial whose coefficients have greatest common divisor $1$ is called primitive. A primitive polynomial is reducible in $\mathbb{Q}[x]$ if and only if it is a product of two nonconstant primitive integral polynomials. Let $K/\mathbb{Q}_p$ be a finite extension, normalize $v_p(p)=1$, and set $v_p(0)=+\infty$. For a nonzero polynomial $f=\sum_j f_jx^j$, define
\[
\|f\|_p=\max_j p^{-v_p(f_j)}.
\]
Then $\|fg\|_p=\|f\|_p\|g\|_p$. The same conclusion holds for multivariate polynomials when the norm is defined as the maximum of the absolute values of all coefficients.
\end{theorem}

For polynomials with integral coefficients and at least one unit coefficient, the latter assertion follows because reduction modulo the maximal ideal gives an integral domain. Multiplying general polynomials by scalars to normalize them gives the stated multiplicativity.

\begin{theorem}[Valuations of factorials]
For a prime $p$ and a nonnegative integer $N$,
\[
v_p(N!)=\sum_{r\geq1}\left\lfloor\frac{N}{p^r}\right\rfloor.
\]
Here $\lfloor y\rfloor$ denotes the greatest integer not exceeding the real number $y$. If $s_2(N)$ denotes the sum of the binary digits of $N$, with $s_2(0)=0$, then $v_2(N!)=N-s_2(N)$. If $N<p^2$, then $v_p(N!)=\lfloor N/p\rfloor$.
\end{theorem}

\begin{theorem}[Newton polygons]
Let $K$ be a complete discretely valued field, with a valuation $v$ satisfying $v(K^\times)=\mathbb{Z}$. For $f(X)=\sum_{j=0}^d f_jX^j\in K[X]$, first remove any zero roots. Its Newton polygon is the lower convex hull of the points $(j,v(f_j))$ for which $f_j\neq0$. If an edge has slope $-\eta$ and horizontal length $q$, then exactly $q$ roots satisfy $v(\alpha)=\eta$, counted with multiplicity. The roots corresponding to a given slope form a factor over $K$. If the absolute value of the vertical change along this edge is $1$, then the corresponding factor is irreducible over $K$.
\end{theorem}

For the statements concerning root valuations and factorization, see \cite[Section~6.4, pp.~212--230]{gouvea}. The final assertion also has a direct explanation: in this case $\eta=\pm1/q$. If a monic irreducible factor belonging to this slope has degree $d_0$, then the valuation of its constant term is $d_0\eta\in\mathbb{Z}$, so $q\mid d_0$. Since there are only $q$ roots belonging to this slope, no further factorization is possible.

\begin{theorem}[Basic properties of Legendre polynomials]
All roots of $P_n$ are simple and lie in $(-1,1)$. For $n\geq1$,
\begin{align}
(1-x^2)P_n''-2xP_n'+n(n+1)P_n&=0, \label{eq:1.2}\\
(n+1)P_{n+1}&=(2n+1)xP_n-nP_{n-1}, \label{eq:1.3}\\
(1-x^2)P_n'&=n(P_{n-1}-xP_n). \label{eq:1.4}
\end{align}
Moreover,
\begin{align}
\sum_{n\geq0}P_n(x)z^n&=(1-2xz+z^2)^{-1/2}, \label{eq:1.5}\\
\int_{-1}^1 q(x)P_n(x)\,dx&=0\qquad(\deg q<n), \label{eq:1.6}\\
|P_n(x)|&\leq1\qquad(-1\leq x\leq1). \label{eq:1.7}
\end{align}
Equation~(1.5) is used as an identity of formal power series.
\end{theorem}

These standard formulas and the location of the roots may be found in \cite{dlmf}: the recurrence and derivative formulas are given in Section~14.10, the generating function in Eq.~14.7.19, orthogonality and properties of the roots in Sections~18.2 and~18.3, and the bound on the interval in Section~18.14. Orthogonality also follows from~(1.1) by integrating by parts $n$ times.

\begin{theorem}[Hadamard's inequality]
If the column vectors of a real or complex $k\times k$ matrix are $w_1,\ldots,w_k$, then
\[
|\det(w_1,\ldots,w_k)|\leq\prod_{i=1}^k\|w_i\|_2.
\]
\end{theorem}

This follows from Gram--Schmidt orthogonalization and the fact that orthogonal projection does not increase length.

\begin{theorem}[Hanson {\cite[Theorem~1]{hanson}}]
For every positive integer $N$,
\[
\operatorname{lcm}(1,2,\ldots,N)<3^N.
\]
Consequently, if $\vartheta(y)=\sum_{p\leq y}\log p$, where the sum is over primes and an empty sum is taken to be zero, then
\begin{equation}
\vartheta(y)<y\log3\qquad(y>0).
\label{eq:1.8}
\end{equation}
\end{theorem}

For $y\geq1$, the product of the primes not exceeding $\lfloor y\rfloor$ divides the least common multiple above, so taking logarithms gives~(1.8). The case $0<y<1$ is immediate. Throughout the paper, $\log$ denotes the natural logarithm.

\begin{theorem}[Known irreducibility results {\cite[Theorem~1.9(b) and Corollary~3.4(b)]{cullinan}}]
\leavevmode
\begin{enumerate}
\item For $1\le j\le60$, both $P_{2j}(x)$ and $P_{2j+1}(x)/x$ are irreducible in $\mathbb Q[x]$.
\item If $p$ is an odd prime and $n\ge2$ satisfies $p-4\le n\le p+3$, then $P_n(x)$ for even $n$, and $P_n(x)/x$ for odd $n$, are irreducible in $\mathbb Q[x]$.
\end{enumerate}
\end{theorem}

The first assertion is the irreducibility consequence of the Galois-group computation in \cite[Theorem~1.9(b)]{cullinan}, with original degree $2j$ or $2j+1$. The second is \cite[Corollary~3.4(b)]{cullinan}, a result of Holt completed by Wahab. These results will be used in Section~6.

\begin{theorem}[Wahab's $2$-adic Newton polygon]
Let the binary expansion of the positive integer $n$ be
\[
n=q_1+\cdots+q_s,\qquad q_1>\cdots>q_s,
\]
where each $q_i$ is a power of $2$. Put $K_0=0$ and $K_i=q_1+\cdots+q_i$. Then the vertices of the $2$-adic Newton polygon of $P_n(2X+1)$ are exactly
\[
(0,0),(K_1,1),\ldots,(K_s,s).
\]
Thus the $i$th edge has horizontal length $q_i$, vertical increment $1$, and slope $1/q_i$.
\end{theorem}

This result is given by Wahab \cite[Theorem~3.1]{wahab} and restated by Cullinan and Hajir \cite[Theorem~7.3]{cullinan}.

\subsection{Outline of the proof}

We argue by contradiction, ruling out a nontrivial factorization of a Legendre polynomial by obtaining upper and lower bounds for the same resultant.

Write $n=2m+\varepsilon$, where $\varepsilon\in\{0,1\}$, and put
\[
\kappa_n=2^{v_2(n!)},\qquad M_n(x)=\frac{\kappa_nP_n(x)}{x^\varepsilon}.
\]
Section~2 proves that $M_n$ is a primitive integral polynomial and that each of its rational factors is an even polynomial. Hence, if $M_n$ is reducible, we may write
\[
M_n(x)=A(x^2)B(x^2),\qquad H(t)=A(t)B(t),
\]
\[
1\leq k=\deg A\leq e=\deg B,\qquad k+e=m,
\]
where $A,B\in\mathbb{Z}[t]$. Since $H$ has no repeated roots, the resultant
\[
R=\operatorname{Res}_t(A,B)
\]
is a nonzero integer. We estimate its odd part
\[
\mathcal{R}=\frac{|R|}{2^{v_2(R)}}.
\]
The upper and lower bounds below are both established for $n\geq64$.

The upper bound comes from valuations at odd primes. We use the classical auxiliary polynomial
\[
U_n(x)=\sum_{j=1}^n\frac{P_{j-1}(x)P_{n-j}(x)}{j}
\]
and its identity
\[
(1-x^2)(P_n'U_n-P_nU_n')=1-P_n^2.
\]
At a root $\alpha$ of $P_n$, this becomes
\[
(1-\alpha^2)P_n'(\alpha)U_n(\alpha)=1.
\]
Section~3 uses this identity and orthogonality to control valuations of derivatives at the roots. Expressing both the derivatives and the resultant as products of root differences then yields
\[
v_p(R)\leq k\lfloor\log_p(2n-1)\rfloor
\qquad(p\text{ an odd prime}).
\]
Here $\log_p y=\log y/\log p$. In addition, the three-term recurrence, the differential equation, and divisibility properties of the coefficients give
\[
v_p(R)=0\qquad(p>m).
\]
Section~4 shows that, with $q=\lfloor\sqrt{2n-1}\rfloor$,
\[
\mathcal R^{1/k}\le
\operatorname{lcm}(1,\ldots,m)\operatorname{lcm}(1,\ldots,q)
<3^{m+q}.
\]
Thus
\[
\frac{\log\mathcal R}{k}
<\bigl(m+\sqrt{2n-1}\bigr)\log3.
\]

The lower bound comes from derivatives at real roots. Let $a>0$ be the leading coefficient of $A$, and let its roots be $t_1,\ldots,t_k$. At these roots, $H'=A'B$. Multiplying these identities and using the root-product formula for the resultant gives
\[
|R|=\frac{a^{e-k}\displaystyle\prod_{i=1}^k|H'(t_i)|}
{\displaystyle\prod_{1\leq i<j\leq k}(t_i-t_j)^2}.
\]
Thus we need a lower bound for the derivatives in the numerator and an upper bound for the product of root differences in the denominator. In Section~5, the Legendre differential equation gives
\[
\frac{d}{dx}\bigl((1-x^2)P_n'(x)^2+n(n+1)P_n(x)^2\bigr)
=2xP_n'(x)^2.
\]
Evaluating the expression at $x=0$ and at a positive root of $P_n$ yields $|H'(t_i)|>\kappa_n\sqrt m$. On the other hand, Chebyshev polynomials and Theorem~1.6 give
\[
\prod_{1\leq i<j\leq k}(t_i-t_j)^2\leq\frac{k^k}{2^{2(k-1)^2}}.
\]
Together with the estimates for the leading coefficients and $v_2(R)$ in that section, this yields
\[
\frac{\log\mathcal R}{k}>
m\log4-\log\frac{4(m+1)^2}{\sqrt m}.
\]
Finally, Section~6 proves that for every $n\ge256$,
\[
m\log(4/3)>
\sqrt{4m+1}\log3+\log\frac{4(m+1)^2}{\sqrt m}.
\]
The lower bound then strictly exceeds the upper bound, so the hypothetical factorization cannot exist. Theorem~1.8 handles all but fifteen of the remaining degrees; explicit integer comparisons using the same bounds handle those fifteen, proving the main theorem.

\section{Reduction and construction of the resultant}

In this section we turn a hypothetical rational factorization into a factorization by primitive integral polynomials in the squared variable, and construct the resultant to be estimated later. Fix an integer $n\geq2$ and write
\[
n=2m+\varepsilon,\qquad m=\lfloor n/2\rfloor,\qquad\varepsilon\in\{0,1\}.
\]
Define
\begin{equation}
\kappa_n=2^{v_2(n!)},\qquad M_n(x)=\kappa_n\frac{P_n(x)}{x^\varepsilon}.
\label{eq:2.1}
\end{equation}
Here $v_2$ is the $2$-adic valuation normalized by $v_2(2)=1$, and also denotes its extension to an algebraic closure of $\mathbb{Q}_2$; $\mathbb{Q}_2$ denotes the field of $2$-adic numbers. We write $\operatorname{lc}(f)$ for the leading coefficient of a nonzero polynomial $f$.

\subsection{Integral polynomials and their factors}

\begin{lemma}
The polynomial $M_n$ is a primitive integral polynomial with odd leading coefficient.
\end{lemma}

\begin{proof}
Rodrigues' formula gives
\begin{equation}
P_n(x)=2^{-n}\sum_{j=0}^m(-1)^j
\frac{(2n-2j)!}{j!(n-j)!(n-2j)!}x^{n-2j}.
\label{eq:2.2}
\end{equation}
The factorial quotient equals $\binom{2n-2j}{n-j}\binom{n-j}{j}$, so the denominator of every coefficient is a power of $2$. Denote the coefficient of $x^{n-2j}$ by $d_{n,j}$. By Theorem~1.3,
\begin{align*}
v_2(d_{n,j})
&=-n+v_2((2n-2j)!)-v_2(j!)-v_2((n-j)!)-v_2((n-2j)!)\\
&=-n+s_2(j)+s_2(n-2j).
\end{align*}
Here we used $s_2(2n-2j)=s_2(n-j)$. Carries in binary addition cannot increase the digit sum, so
\[
s_2(n)=s_2(2j+n-2j)\leq s_2(j)+s_2(n-2j).
\]
Since $v_2(\kappa_n)=v_2(n!)=n-s_2(n)$, it follows that $v_2(\kappa_nd_{n,j})\geq0$, and hence $M_n\in\mathbb{Z}[x]$. For $j=0$, this valuation is zero, so the leading coefficient is odd. Moreover, $M_n(1)=\kappa_n$ is a power of $2$. Any positive integer dividing all coefficients of $M_n$ divides both $\kappa_n$ and the odd leading coefficient, and therefore must equal $1$.
\end{proof}

\begin{lemma}
Every rational factor of $M_n$ is an even polynomial. If $n=q_1+\cdots+q_s$, where $q_1>\cdots>q_s$ are distinct powers of $2$, then the degree of every rational factor is a sum of some of the $q_i\geq2$; distinct factors in the same factorization select disjoint sets of these binary place values.
\end{lemma}

\begin{proof}
By Theorem~1.9, the edges of the $2$-adic Newton polygon of $P_n(2X+1)$ have horizontal lengths $q_i$ and vertical increments $1$. By Theorem~1.4, each edge corresponds to an irreducible factor of degree $q_i$ over $\mathbb{Q}_2$, whose roots satisfy $v_2(X)=-1/q_i$. Returning to the variable $x=2X+1$, the corresponding roots satisfy
\begin{equation}
v_2(1-x)=1-\frac1{q_i}.
\label{eq:2.3}
\end{equation}
When $q_i\geq2$, the right-hand side is less than $v_2(2)=1$, and hence
\[
v_2(1+x)=v_2\bigl(2-(1-x)\bigr)=1-\frac1{q_i}.
\]
The root set of $P_n$ is invariant under negation, and this equality shows that negation preserves the root set corresponding to each edge. These nonzero roots occur in opposite pairs, so the corresponding monic irreducible factor is an even polynomial.

If $q_i=1$, then $n$ is odd. This group contains only one root, and $x=0$ satisfies~(2.3), so it is precisely the zero root removed in forming $M_n$. Thus, up to a nonzero constant, $M_n$ is the product of the even irreducible factors above in $\mathbb{Q}_2[x]$. Any rational factor must consist of some complete irreducible factors from this product, and is therefore still even. Distinct factors select disjoint sets of local irreducible factors, which also proves the assertion concerning degrees.
\end{proof}

\subsection{Construction of the resultant}

By Lemma~2.1 and the parity of $P_n$, the polynomial $M_n$ is an even integral polynomial of degree $2m$. Thus there is a unique $H\in\mathbb{Z}[t]$ such that $M_n(x)=H(x^2)$, with $\deg H=m$.

Now suppose that $M_n$ is reducible in $\mathbb{Q}[x]$. By Theorem~1.2 and Lemma~2.2, we may choose nonconstant primitive integral polynomials $F,G,A,B$ with positive leading coefficients such that
\begin{equation}
M_n(x)=F(x)G(x)=A(x^2)B(x^2),\qquad H(t)=A(t)B(t).
\label{eq:2.4}
\end{equation}
After interchanging the two factors if necessary, put
\[
k=\deg A\leq e=\deg B=m-k,\qquad
a=\operatorname{lc}(A)>0,\qquad b=\operatorname{lc}(B)>0.
\]
Then $1\leq k\leq m/2$, and $ab=\operatorname{lc}(M_n)$ implies that both $a$ and $b$ are odd. By Theorem~1.5, all roots of $H$ are simple and lie in $(0,1)$. Denote the roots of $A$ by $t_1,\ldots,t_k$ and those of $B$ by $u_1,\ldots,u_e$.

Define the resultant with respect to the variable $t$ by
\begin{equation}
R=\operatorname{Res}_t(A,B)
=a^eb^k\prod_{i=1}^k\prod_{j=1}^e(t_i-u_j)
=a^e\prod_{i=1}^k B(t_i).
\label{eq:2.5}
\end{equation}
It equals the determinant of the Sylvester matrix formed from the coefficients of $A$ and $B$, so $R\in\mathbb{Z}$. Since $H=AB$ has no repeated roots, the two root sets are disjoint and $R\neq0$. In the original variable $x$, the root-product formula for the resultant gives
\begin{equation}
\operatorname{Res}_x(F,G)
=a^{2e}\prod_{i=1}^k G(\sqrt{t_i})G(-\sqrt{t_i})
=a^{2e}\prod_{i=1}^k B(t_i)^2=R^2.
\label{eq:2.6}
\end{equation}
Thus valuations obtained from the resultant in the original variable must be divided by $2$ to obtain the valuations of $R$.

\subsection{The odd part of the resultant}

Define
\begin{equation}
\mathcal{R}=\frac{|R|}{2^{v_2(R)}}.
\label{eq:2.7}
\end{equation}
Since $R$ is a nonzero integer, $\mathcal{R}$ is a positive odd integer, and
\begin{equation}
\log\mathcal{R}=\sum_{\substack{p\ \mathrm{prime}\\p\ \mathrm{odd}}}v_p(R)\log p,
\label{eq:2.8}
\end{equation}
where $v_p(R)$ denotes the exponent of the prime $p$ in $R$. Sections~3 and~4 estimate $\log\mathcal{R}$ from above by summing the prime valuations on the right. Section~5 estimates $|R|$ using the real roots and subtracts the $2$-adic contribution to obtain a lower bound for $\log\mathcal{R}$. Finally, Section~6 compares the two bounds to rule out the factorization~(2.4).

\section{Two local upper bounds for the valuation of the resultant}

In this section we bound both the exponent of an odd prime in the resultant
and the range of odd primes that can divide it. Fix $n\ge64$ and retain the
assumed factorization from Section~2:
\[
 n=2m+\varepsilon,\qquad \varepsilon\in\{0,1\},\qquad
 M_n(x)=\kappa_nP_n(x)/x^\varepsilon=F(x)G(x)=A(x^2)B(x^2).
\]
Here $F,G,A,B$ are primitive polynomials with integer coefficients,
$k=\deg A\le e=\deg B=m-k$, and we write
\[
 a=\operatorname{lc}(A),\qquad b=\operatorname{lc}(B),\qquad
 R=\operatorname{Res}_t(A,B)\ne0.
\]
Thus $\deg F=2k$ and $\deg G=2e$, with leading coefficients $a$ and $b$,
respectively. Throughout this section $p$ is an odd prime, and $v_p$ denotes
the valuation normalized by $v_p(p)=1$, as well as its extension to
$\overline{\mathbb Q}_p$; we set $v_p(0)=+\infty$. For a finite extension
$K/\mathbb Q_p$, write
$\mathcal O_K=\{z\in K:v_p(z)\ge0\}$ for its ring of integers; an element
of valuation zero is called a unit. Put $\log_p y=\log y/\log p$.
The basic formulas for Legendre polynomials used below are listed in
Theorem~1.5.

\subsection{The first local upper bound}

We first estimate the valuation of the derivative at a root. For roots of
nonnegative valuation, we apply the Wronskian identity directly. For roots
of negative valuation, we pass to the reciprocal variable and use
orthogonality in a finite polynomial calculation. Expressing the
valuations of the derivative and the resultant as sums of nonnegative
root-difference valuations then gives the following result.

\begin{theorem}[First local upper bound]
For every odd prime $p$,
\begin{equation}
 v_p(R)\le k\lfloor\log_p(2n-1)\rfloor.
 \label{eq:3.1}
\end{equation}
\end{theorem}

\begin{lemma}[Wronskian identity]
Define
\begin{equation}
 U_n(x)=\sum_{j=1}^n\frac{P_{j-1}(x)P_{n-j}(x)}{j},
 \qquad h_n=\sum_{j=1}^n\frac1j.
 \label{eq:3.2}
\end{equation}
Here $U_n$ is the polynomial part $W_{n-1}$ of the classical Legendre
function of the second kind; see \cite[Eq.~14.7.4]{dlmf} for this expression.
Then $\deg U_n\le n-1$, and
\begin{equation}
 |U_n(x)|\le h_n\qquad(-1\le x\le1),
 \label{eq:3.3}
\end{equation}
\begin{equation}
 (1-x^2)\bigl(P_n'(x)U_n(x)-P_n(x)U_n'(x)\bigr)=1-P_n(x)^2.
 \label{eq:3.4}
\end{equation}
In particular, if $P_n(\alpha)=0$, then
\begin{equation}
 (1-\alpha^2)P_n'(\alpha)U_n(\alpha)=1.
 \label{eq:3.5}
\end{equation}
Moreover,
\begin{equation}
 U_n(x)=\frac12\int_{-1}^1\frac{P_n(x)-P_n(t)}{x-t}\,dt.
 \label{eq:3.6}
\end{equation}
The integrand is a polynomial in $x$ and $t$, so no singular integral is
involved.
\end{lemma}

\begin{proof}
The degree bound and the estimate on the real interval follow directly
from the definition and the bound $|P_j(x)|\le1$ in Theorem~1.5.
For real $x>1$, put
\[
 L(x)=\frac12\log\frac{x+1}{x-1}.
\]
Use the Legendre function of the second kind $Q_n$ determined by
\[
 Q_n(x)=\frac12\int_{-1}^1\frac{P_n(t)}{x-t}\,dt.
\]
This integral formula is given in
\cite[Eqs.~14.12.13 and 14.7.6]{dlmf}. With the same normalization,
\cite[Eqs.~14.7.4, 14.7.7, and 14.2.10]{dlmf} gives
\[
 Q_n=P_nL-U_n,\qquad
 P_nQ_n'-P_n'Q_n=\frac1{1-x^2}.
\]
Since $L'(x)=1/(1-x^2)$, substitution yields
\[
 \frac1{1-x^2}=\frac{P_n(x)^2}{1-x^2}
       +P_n'(x)U_n(x)-P_n(x)U_n'(x).
\]
Multiplication by $1-x^2$ gives (3.4). Also,
\[
 L(x)=\frac12\int_{-1}^1\frac{dt}{x-t},
\]
so $U_n=P_nL-Q_n$ gives (3.6). Both sides of each of these two identities
are polynomials with rational coefficients; their validity for $x>1$
therefore proves the corresponding polynomial identities. Finally,
substituting $P_n(\alpha)=0$ into (3.4) gives (3.5).
\end{proof}

\begin{lemma}[Valuation bounds for the derivative at a root]
Let $p$ be an odd prime, put $L_p=\lfloor\log_p(2n-1)\rfloor$, and define
the reciprocal polynomial $\widehat P_n(Y)=Y^nP_n(1/Y)$.
If $P_n(\alpha)=0$, then
\begin{equation}
 v_p(\alpha)\ge0\quad\Longrightarrow\quad
 v_p\bigl(P_n'(\alpha)\bigr)\le\lfloor\log_p n\rfloor\le L_p,
 \label{eq:3.7}
\end{equation}
\begin{equation}
 v_p(\alpha)<0\quad\Longrightarrow\quad
 v_p\bigl(\widehat P_n'(\alpha^{-1})\bigr)
 \le L_p-v_p(\alpha^{-1})<L_p.
 \label{eq:3.8}
\end{equation}
\end{lemma}

\begin{proof}
First suppose that $v_p(\alpha)\ge0$. Then $v_p(1-\alpha^2)=0$.
Otherwise $v_p(\alpha-\eta)>0$ for some $\eta\in\{1,-1\}$.
Since $P_n\in\mathbb Z_p[x]$, the polynomial $P_n(x)-P_n(\eta)$ is
divisible by $x-\eta$, with quotient in $\mathbb Z_p[x]$. Hence
$v_p(P_n(\alpha)-P_n(\eta))>0$, contrary to $P_n(\alpha)=0$ and
$P_n(\eta)=\eta^n$. From $P_j\in\mathbb Z_p[x]$ and (3.2), we obtain
\[
 v_p\bigl(U_n(\alpha)\bigr)
 \ge-\max_{1\le j\le n}v_p(j)=-\lfloor\log_p n\rfloor.
\]
Taking valuations in (3.5) proves (3.7).

Now suppose that $v_p(\alpha)<0$, and write
\[
 \beta=\alpha^{-1},\qquad \rho=v_p(\beta)>0,\qquad
 \widehat U_n(Y)=Y^{n-1}U_n(1/Y).
\]
Since $P_n(\alpha)=0$,
\[
 \widehat P_n'(\beta)=-\beta^{n-2}P_n'(\alpha),\qquad
 \widehat U_n(\beta)=\beta^{n-1}U_n(\alpha).
\]
Thus (3.5) becomes
\begin{equation}
 (1-\beta^2)\widehat P_n'(\beta)\widehat U_n(\beta)=\beta^{2n-1}.
 \label{eq:3.9}
\end{equation}
Since $v_p(1-\beta^2)=0$, it suffices to prove that
\begin{equation}
 v_p\bigl(\widehat U_n(\beta)\bigr)\ge2n\rho-L_p.
 \label{eq:3.10}
\end{equation}

Choose a finite extension $K/\mathbb Q_p$ containing $\beta$.
As $P_n(1/\beta)=0$,
\begin{equation}
 q(t)=\frac{P_n(t)}{1-\beta t}
 \label{eq:3.11}
\end{equation}
is a polynomial of degree $n-1$ in $K[t]$. Its coefficients all lie in
$\mathcal O_K$. Indeed, write $P_n(t)=\sum_{j=0}^n c_jt^j$ and
$q(t)=\sum_{j=0}^{n-1}b_jt^j$. Comparing coefficients in
$P_n=(1-\beta t)q$ gives
\[
 b_0=c_0,\qquad b_j=c_j+\beta b_{j-1}\quad(1\le j\le n-1).
\]
Since $c_j\in\mathbb Z_p$ and $\beta\in\mathcal O_K$, induction gives
$b_j\in\mathcal O_K$.

Define a $K$-linear map $I:K[t]\longrightarrow K$ by
\begin{equation}
 I(t^j)=\begin{cases}
 0,&j\text{ odd},\\
 1/(j+1),&j\text{ even}.
 \end{cases}
 \label{eq:3.12}
\end{equation}
For polynomials with rational coefficients,
$I(f)=\frac12\int_{-1}^1f(t)\,dt$.
Orthogonality therefore gives $I(t^jP_n(t))=0$ for $0\le j<n$.
Evaluating the polynomial identity (3.6) at $x=1/\beta$ gives
\begin{equation}
 \widehat U_n(\beta)=-\beta^nI(q).
 \label{eq:3.13}
\end{equation}
On the other hand, the finite geometric sum gives
\[
 q(t)=P_n(t)\sum_{j=0}^{n-1}(\beta t)^j+\beta^nt^nq(t).
\]
Apply $I$ to both sides. Every term in the first summand vanishes by
orthogonality, so
\[
 I(q)=\beta^nI(t^nq),\qquad
 \widehat U_n(\beta)=-\beta^{2n}I(t^nq).
\]
Moreover,
\[
 I(t^nq)=\sum_{\substack{0\le j\le n-1\\n+j\text{ even}}}
                  \frac{b_j}{n+j+1}.
\]
Here $b_j\in\mathcal O_K$, and the denominator of each nonzero term is a
positive integer at most $2n-1$. Indeed, $n+j$ must be even, so
$n+j\le2n-2$; the potential denominator $2n$ corresponds to an odd power,
which is annihilated by $I$. Consequently,
$v_p(I(t^nq))\ge-L_p$, proving (3.10).
Finally, taking valuations in (3.9) yields
\[
 v_p\bigl(\widehat P_n'(\beta)\bigr)
 =(2n-1)\rho-v_p\bigl(\widehat U_n(\beta)\bigr)\le L_p-\rho,
\]
which is (3.8).
\end{proof}

\begin{lemma}[Nonnegative root-difference decompositions of the resultant and derivative]
Let $Z_n,Z_F,Z_G$ be the sets of roots of $P_n,F,G$, respectively, in
$\overline{\mathbb Q}_p$. For distinct roots $\alpha,\gamma\in Z_n$, define
\begin{equation}
 \delta_p(\alpha,\gamma)=v_p(\alpha-\gamma)
       -\min\{v_p(\alpha),0\}-\min\{v_p(\gamma),0\}.
 \label{eq:3.14}
\end{equation}
Then $\delta_p(\alpha,\gamma)\ge0$, and
\begin{equation}
 2v_p(R)=\sum_{\alpha\in Z_F}\sum_{\gamma\in Z_G}\delta_p(\alpha,\gamma).
 \label{eq:3.15}
\end{equation}
For every $\alpha\in Z_n$,
\begin{equation}
 \sum_{\substack{\gamma\in Z_n\\\gamma\ne\alpha}}
       \delta_p(\alpha,\gamma)
 =\begin{cases}
 v_p\bigl(P_n'(\alpha)\bigr),&v_p(\alpha)\ge0,\\
 v_p\bigl(\widehat P_n'(\alpha^{-1})\bigr),&v_p(\alpha)<0.
 \end{cases}
 \label{eq:3.16}
\end{equation}
In addition, for $Q=F,G$,
\begin{equation}
 v_p(\operatorname{lc}Q)=-\sum_{Q(\gamma)=0}\min\{v_p(\gamma),0\},
 \label{eq:3.17}
\end{equation}
\begin{equation}
 v_p\bigl(Q(0)\bigr)=\sum_{Q(\gamma)=0}\max\{v_p(\gamma),0\}.
 \label{eq:3.18}
\end{equation}
\end{lemma}

\begin{proof}
According to the signs of the two root valuations, the definition becomes
\begin{equation}
 \delta_p(\alpha,\gamma)=\begin{cases}
 v_p(\alpha-\gamma),&v_p(\alpha),v_p(\gamma)\ge0,\\
 v_p(\alpha^{-1}-\gamma^{-1}),&v_p(\alpha),v_p(\gamma)<0,\\
 0,&\text{one valuation is nonnegative and the other is negative}.
 \end{cases}
 \label{eq:3.19}
\end{equation}
The second line uses
$\alpha^{-1}-\gamma^{-1}=(\gamma-\alpha)/(\alpha\gamma)$.
The third uses the fact that the valuation of the difference of two
elements of unequal valuations is the smaller of those valuations.
Each right-hand side is nonnegative.

Choose a finite extension $K/\mathbb Q_p$ containing all the roots.
For $Q=P_n,F,G$, the Gauss norm is $1$: the coefficients of $P_n$ lie in
$\mathbb Z_p$ and $P_n(1)=1$, while $F$ and $G$ are primitive integer
polynomials. Factor $Q$ into linear factors and apply Theorem~1.2 to obtain
\[
 1=\|Q\|_p=p^{-v_p(\operatorname{lc}Q)}
        \prod_{Q(\gamma)=0}\max\{1,p^{-v_p(\gamma)}\}.
\]
Taking logarithms gives
\begin{equation}
 v_p(\operatorname{lc}Q)=-\sum_{Q(\gamma)=0}\min\{v_p(\gamma),0\}.
 \label{eq:3.20}
\end{equation}
This proves (3.17). Since neither $F$ nor $G$ has a zero root, the
root-product formula for the constant term further gives
\[
 v_p\bigl(Q(0)\bigr)=v_p(\operatorname{lc}Q)+\sum_{Q(\gamma)=0}v_p(\gamma)
 =\sum_{Q(\gamma)=0}\max\{v_p(\gamma),0\},
\]
proving (3.18).

The root-product formula for the resultant gives
\begin{align*}
 v_p\bigl(\operatorname{Res}_x(F,G)\bigr)
 &=2e\,v_p(a)+2k\,v_p(b)
       +\sum_{\alpha\in Z_F}\sum_{\gamma\in Z_G}v_p(\alpha-\gamma)\\
 &=\sum_{\alpha\in Z_F}\sum_{\gamma\in Z_G}
   \bigl(v_p(\alpha-\gamma)-\min\{v_p(\alpha),0\}-\min\{v_p(\gamma),0\}\bigr).
\end{align*}
The second equality uses (3.20). Together with
$\operatorname{Res}_x(F,G)=R^2$ from Section~2, this proves (3.15).

Finally, all roots of $P_n$ are simple, so
\[
 v_p\bigl(P_n'(\alpha)\bigr)=v_p(\operatorname{lc}P_n)
       +\sum_{\substack{\gamma\in Z_n\\\gamma\ne\alpha}}v_p(\alpha-\gamma).
\]
Substituting (3.20) for the valuation of the leading coefficient and
rearranging gives
\[
 \sum_{\substack{\gamma\in Z_n\\\gamma\ne\alpha}}\delta_p(\alpha,\gamma)
 =v_p\bigl(P_n'(\alpha)\bigr)-(n-2)\min\{v_p(\alpha),0\}.
\]
If $v_p(\alpha)\ge0$, this is the first case of (3.16).
If $v_p(\alpha)<0$, use
\[
 \widehat P_n'(\alpha^{-1})=-\alpha^{2-n}P_n'(\alpha)
\]
to identify the right-hand side with
$v_p(\widehat P_n'(\alpha^{-1}))$, proving the second case.
\end{proof}

\par\medskip\noindent\textbf{Proof of Theorem 3.1.}
Fix $\alpha\in Z_F$. Since $Z_G\subset Z_n\setminus\{\alpha\}$ and
every $\delta_p$ in Lemma~3.4 is nonnegative, Lemma~3.3 gives
\[
 \sum_{\gamma\in Z_G}\delta_p(\alpha,\gamma)
 \le\sum_{\substack{\gamma\in Z_n\\\gamma\ne\alpha}}\delta_p(\alpha,\gamma)
 \le\lfloor\log_p(2n-1)\rfloor.
\]
Summing over the $2k$ roots of $F$ and using (3.15), we obtain
\[
 2v_p(R)\le2k\lfloor\log_p(2n-1)\rfloor.
\]
Dividing by $2$ proves (3.1).\hfill$\square$

\subsection{The second local upper bound}

We prove that an odd prime $p>m=\lfloor n/2\rfloor$ cannot divide the
resultant $R$.

\begin{theorem}[Second local upper bound]
For every odd prime $p>m$,
\begin{equation}
 v_p(R)=0.
 \label{eq:3.21}
\end{equation}
\end{theorem}

The proof has two steps. We first treat roots of valuation zero, and then
use the constant term and leading coefficient to treat roots of positive
and negative valuation.

\begin{lemma}
Let $p>m$ be an odd prime. If
\[
 P_n(\alpha)=0,\qquad v_p(\alpha)=0,
\]
then $v_p(P_n'(\alpha))=0$.
\end{lemma}

\begin{proof}
For a polynomial $f$ with coefficients in $\mathbb Z_p$, let
$\overline f\in\mathbb F_p[x]$ denote its coefficientwise reduction
modulo $p$. This notation applies to Legendre polynomials because the
denominators of their coefficients are powers of $2$.

Since $p>m$, we have $n<2p$. Write
\[
 n=\lambda p+r,\qquad \lambda\in\{0,1\},\qquad 0\le r<p.
\]
We first prove that
\begin{equation}
 \overline P_n(x)=x^{\lambda p}\overline P_r(x).
 \label{eq:3.22}
\end{equation}
By the coefficient formula (2.2),
\[
 P_p(x)=2^{-p}\sum_{j=0}^{(p-1)/2}(-1)^j
             \binom pj\binom{2p-2j}{p}x^{p-2j}.
\]
For $j\ge1$, the factor $\binom pj$ is divisible by $p$, while the
leading coefficient satisfies
\[
 \binom{2p}{p}=2\prod_{j=1}^{p-1}\frac{p+j}{j}\equiv2\pmod p,
 \qquad 2^p\equiv2\pmod p.
\]
Thus $\overline P_p=x^p$. The three-term recurrence (1.3) first gives
$\overline P_{p+1}=x^{p+1}$ and then gives
\[
 (j+1)\overline P_{p+j+1}
 =(2j+1)x\overline P_{p+j}-j\overline P_{p+j-1},
 \qquad 1\le j\le p-2.
\]
Since $j+1$ is invertible in $\mathbb F_p$, comparison with the
recurrence for $\overline P_j$ proves by induction that
\[
 \overline P_{p+j}=x^p\overline P_j\qquad(0\le j<p).
\]
This proves (3.22).

Next, we show that $\overline P_r$ has no multiple roots. First,
\[
 \overline P_r(1)=1,\qquad \overline P_r(-1)=(-1)^r,
\]
so neither $1$ nor $-1$ is a root. If another root
$\xi\in\overline{\mathbb F}_p$ has multiplicity $s\ge2$, write
\[
 \overline P_r(x)=(x-\xi)^sV(x),\qquad V(\xi)\ne0.
\]
Here $2\le s\le r<p$. Substitute into the reduction of the differential
equation (1.2) in $\mathbb F_p[x]$:
\[
 (1-x^2)\overline P_r''-2x\overline P_r'+r(r+1)\overline P_r=0.
\]
The coefficient of $(x-\xi)^{s-2}$ on the left is
\[
 (1-\xi^2)s(s-1)V(\xi)\ne0,
\]
a contradiction.

Therefore, (3.22) and $(x^{\lambda p})'=0$ in characteristic $p$ give
\[
 \gcd(\overline P_n,\overline P_n')=x^{\lambda p}.
\]
Apply the polynomial Euclidean algorithm and lift the resulting
coefficients to representatives in $\mathbb Z_p$. This gives
$u,v,w\in\mathbb Z_p[x]$ such that
\begin{equation}
 uP_n+vP_n'=x^{\lambda p}+pw.
 \label{eq:3.23}
\end{equation}
Evaluation at the root $\alpha$ yields
\[
 v(\alpha)P_n'(\alpha)=\alpha^{\lambda p}+pw(\alpha).
\]
Since $v_p(\alpha)=0$, the right-hand side has valuation zero. Both
factors on the left have nonnegative valuation, so
$v_p(P_n'(\alpha))=0$.
\end{proof}

\begin{lemma}
Let $p>m$ be an odd prime, and let $c$ be the leading coefficient of
$M_n$. Then
\begin{equation}
 v_p(c)\le1,\qquad v_p\bigl(M_n(0)\bigr)\le1.
 \label{eq:3.24}
\end{equation}
Consequently, $F$ and $G$ cannot both have a root of positive valuation,
nor can they both have a root of negative valuation.
\end{lemma}

\begin{proof}
Since $n\ge64$ and $p>m$,
\[
 2n\le4m+2<(m+1)^2\le p^2.
\]
By the coefficient formula (2.2) for $M_n$ and Theorem~1.3,
\[
 v_p(c)=\left\lfloor\frac{2n}{p}\right\rfloor
          -2\left\lfloor\frac n p\right\rfloor\in\{0,1\}.
\]
Multiplication by $\kappa_n$ does not change valuations at odd primes.
For the constant term,
\[
 v_p\bigl(M_n(0)\bigr)=\begin{cases}
 \displaystyle\left\lfloor\frac{2m}{p}\right\rfloor,&n=2m,\\[6pt]
 \displaystyle\left\lfloor\frac{2m+2}{p}\right\rfloor
             -\left\lfloor\frac{m+1}{p}\right\rfloor,&n=2m+1.
 \end{cases}
\]
The first case is clearly at most $1$. In the second case, if $p>m+1$,
the two floors are at most $1$ and equal to $0$, respectively; if
$p=m+1$, their difference is $2-1=1$. Thus the valuation of the constant
term is also at most $1$.

We now explain how these bounds constrain the roots. For $Q=F$ or $G$,
the Gauss-norm formula in the proof of Lemma~3.4 gives
\[
 v_p(\operatorname{lc}Q)=-\sum_{Q(\zeta)=0}\min\{v_p(\zeta),0\}.
\]
Combining this with the root-product formula for the constant term gives
\[
 v_p\bigl(Q(0)\bigr)=\sum_{Q(\zeta)=0}\max\{v_p(\zeta),0\}.
\]
Hence a root of positive valuation forces $v_p(Q(0))>0$. Since $Q(0)$
is a nonzero integer, its valuation is then at least $1$. If both $F$
and $G$ had such a root, we would obtain
\[
 v_p\bigl(M_n(0)\bigr)=v_p\bigl(F(0)\bigr)+v_p\bigl(G(0)\bigr)\ge2,
\]
contrary to the bound above.

Similarly, if $Q$ has a root of negative valuation, its leading
coefficient has valuation at least $1$. If both $F$ and $G$ had such a
root, then
\[
 v_p(c)=v_p(\operatorname{lc}F)+v_p(\operatorname{lc}G)\ge2,
\]
again a contradiction.
\end{proof}

\par\medskip\noindent\textbf{Proof of Theorem 3.5.}
Take any root $\alpha$ of $F$ and any root $\gamma$ of $G$.
Retain the notation of Lemma~3.4:
\[
 \delta_p(\alpha,\gamma)=v_p(\alpha-\gamma)
       -\min\{v_p(\alpha),0\}-\min\{v_p(\gamma),0\}.
\]
That lemma shows that $\delta_p(\alpha,\gamma)\ge0$. We prove that
equality must hold.

If one root has valuation zero, assume without loss of generality that
$v_p(\alpha)=0$. By Lemmas~3.6 and 3.4,
\[
 0\le\delta_p(\alpha,\gamma)
 \le\sum_{\substack{P_n(\zeta)=0\\\zeta\ne\alpha}}
                  \delta_p(\alpha,\zeta)
 =v_p\bigl(P_n'(\alpha)\bigr)=0.
\]
If both root valuations are nonzero, Lemma~3.7 rules out their being
both positive or both negative. They must therefore have opposite signs.
Assume without loss of generality that
\[
 v_p(\alpha)>0,\qquad v_p(\gamma)<0.
\]
Then $v_p(\alpha-\gamma)=v_p(\gamma)$, and substitution into the
definition again gives $\delta_p(\alpha,\gamma)=0$.

Thus every pair of roots belonging to $F$ and $G$, respectively,
satisfies $\delta_p(\alpha,\gamma)=0$. Finally, (3.15) gives
\[
 2v_p(R)=\sum_{\alpha\in Z_F}\sum_{\gamma\in Z_G}\delta_p(\alpha,\gamma)=0.
\]
This proves the theorem.\hfill$\square$

\section{A combined upper bound for the odd part of the resultant}

In this section we combine the two local upper bounds from Section~3. The first controls the exponent of each odd prime, while the second restricts the odd primes that can occur to $p\le m$. We place all the required prime powers in two least common multiples and then apply Theorem~1.7.

Throughout this section, we continue to assume that $n\ge64$, $m=\lfloor n/2\rfloor$, $k=\deg A\le\deg B=m-k$, $R=\operatorname{Res}_t(A,B)$, and $\mathcal R=|R|/2^{v_2(R)}$. Here $v_p$ denotes the normalized prime valuation, and $\log$ denotes the natural logarithm.

\begin{theorem}[Combined upper bound]
Under the nontrivial factorization assumption of Section~2,
\begin{equation}
\frac{\log\mathcal R}{k}
<\bigl(m+\sqrt{2n-1}\bigr)\log3.
\label{eq:4.1}
\end{equation}
\end{theorem}

\begin{proof}
Set $N=2n-1$, $q=\lfloor\sqrt N\rfloor$, and
\begin{equation}
D_n=\prod_{\substack{p\le m\\p\text{ odd prime}}}
p^{\lfloor\log_pN\rfloor}.
\label{eq:4.2}
\end{equation}
Theorems~3.1 and~3.5 imply
\begin{equation}
\mathcal R\mid D_n^k,\qquad
\frac{\log\mathcal R}{k}\le\log D_n.
\label{eq:4.3}
\end{equation}
For a positive integer $r$, write $\Lambda(r)=\operatorname{lcm}(1,\ldots,r)$. We claim that
\begin{equation}
D_n\mid\Lambda(m)\Lambda(q).
\label{eq:4.4}
\end{equation}

Fix an odd prime $p\le m$, and put $r=\lfloor\log_pN\rfloor$. If $r=1$, the required factor $p$ already divides $\Lambda(m)$. If $p\ge5$ and $r\ge2$, then
\[
p^{r-1}\le\frac Np\le\frac{4m+1}{5}\le m.
\]
Thus $\Lambda(m)$ contains $p^{r-1}$. Moreover, $p^2\le N$ gives $p\le q$, so $\Lambda(q)$ supplies another factor $p$.

It remains to consider $p=3$. Since $m\ge32$, the case $r\le3$ is already covered by $\Lambda(m)$. If $r\ge4$, then
\[
3^{r-2}\le\frac N9\le\frac{4m+1}{9}<m,
\qquad 3^4\le N.
\]
Consequently $\Lambda(m)$ contains $3^{r-2}$ and $\Lambda(q)$ contains $3^2$. This proves~(4.4).

By Theorem~1.7,
\[
D_n\le\Lambda(m)\Lambda(q)<3^{m+q}
\le3^{m+\sqrt{2n-1}}.
\]
Combining this with~(4.3) proves~(4.1).
\end{proof}

\section{A lower bound for the odd part of the resultant}

We continue to work under the factorization assumption $H=AB$, and write
\[
\deg A=k\le e=\deg B=m-k,\qquad
a=\operatorname{lc}(A)>0,\qquad
\gamma_n=\bigl(n(n+1)P_n(0)^2+P_n'(0)^2\bigr)^{1/2}.
\]
The distribution of the roots of the Legendre polynomials implies that all roots of $H$ are simple and lie in $(0,1)$. We first express the resultant as a product involving derivatives and root differences, and then control the root differences using the length of the interval $(0,1)$.

\begin{theorem}[Lower bound for the odd part of the resultant]
Under the above nontrivial factorization assumption,
\begin{equation}
\frac{\log\mathcal R}{k}>m\log4-\log\frac{4(m+1)^2}{\sqrt m}.
\label{eq:5.1}
\end{equation}
\end{theorem}

To prove this theorem, we first establish the valuation estimate needed to remove the power of $2$, then derive a product relation between the resultant and the derivatives at the roots, and finally estimate the derivatives and the root differences within one factor separately.

\begin{proposition}
Under the factorization~(2.4),
\begin{align}
s_2(m)&=s_2(k)+s_2(e),&\qquad \kappa_n&=2^{2m-s_2(m)},\label{eq:5.2}\\
A(1)&=2^{2k-s_2(k)},& a&>A(1),\label{eq:5.3}\\
v_2(R)&\le e\bigl(2k-s_2(k)\bigr).&&\label{eq:5.4}
\end{align}
\end{proposition}

\begin{proof}
Deleting the units bit in the binary expansion of $n$ and dividing the remaining place values by $2$ yields the binary place values of $m$. Lemma~2.2 shows that these are partitioned into two groups whose sums are $k$ and $e$, respectively, with no place value shared by the two groups. Thus their addition involves no carries, and $s_2(m)=s_2(k)+s_2(e)$. Moreover, $s_2(n)=s_2(m)+\varepsilon$, so $n-s_2(n)=2m-s_2(m)$.

Since all the real roots lie in $(0,1)$, $A(1)$ and $B(1)$ are positive integers with product $\kappa_n$, and are therefore both powers of $2$. To compute the $2$-adic valuations, fix an extension of the $2$-adic valuation to the splitting field of these roots; all valuations of $t_i$ and $u_j$ below are taken with respect to this extension. Let $S$ be the set of binary place values whose sum is $k$. For each $u\in S$, the corresponding $2u$ roots of $F$ satisfy $v_2(1-x)=1-1/(2u)$ by~(2.3). The leading coefficient of $F$ is odd, so the root product formula gives
\[
v_2(A(1))=v_2(F(1))
=\sum_{u\in S}2u\left(1-\frac1{2u}\right)
=\sum_{u\in S}(2u-1)
=2k-s_2(k).
\]
This determines $A(1)$. At the same time,
\[
A(1)=a\prod_{i=1}^{k}(1-t_i)<a,
\]
which proves~(5.3).

Finally, we estimate the power of $2$ in the resultant. A root $t=x^2$ in the squared variable corresponding to the place value $u$ satisfies
\[
v_2(1-t)=v_2(1-x)+v_2(1+x)=2-\frac1u.
\]
Roots belonging to different factors correspond to different place values, so $v_2(1-t_i)\ne v_2(1-u_j)$. Since the valuation of the difference of two elements of unequal valuation is the smaller of their valuations,
\[
v_2(t_i-u_j)
=\min\{v_2(1-t_i),v_2(1-u_j)\}
\le v_2(1-t_i).
\]
Summing over all pairs of roots, and using the fact that $a$ and $b$ are odd together with~(2.5), we obtain
\[
v_2(R)\le e\sum_{i=1}^{k}v_2(1-t_i)
=e\,v_2(A(1))
=e(2k-s_2(k)).
\]
\end{proof}

\begin{lemma}[The resultant and derivatives at the roots]
Let $t_1,\ldots,t_k$ be all the roots of $A$, and put
\[
V_A=\prod_{1\le i<j\le k}(t_i-t_j)^2.
\]
For $k=1$, set $V_A=1$. Then
\begin{equation}
|R|=\frac{a^{e-k}}{V_A}\prod_{i=1}^{k}|H'(t_i)|.
\label{eq:5.5}
\end{equation}
\end{lemma}

\begin{proof}
Since $H=AB$ and $A(t_i)=0$, we have $H'(t_i)=A'(t_i)B(t_i)$. On the other hand, $A(t)=a\prod_{i=1}^{k}(t-t_i)$ gives
\[
\prod_{i=1}^{k}|A'(t_i)|
=a^k\prod_{i=1}^{k}\prod_{j\ne i}|t_i-t_j|
=a^kV_A.
\]
The root product formula for the resultant gives
\[
|R|=a^e\prod_{i=1}^{k}|B(t_i)|.
\]
Substituting these two identities into the product of derivatives yields $\prod_i|H'(t_i)|=a^{k-e}V_A|R|$, and hence~(5.5).
\end{proof}

\begin{lemma}[A lower bound for derivatives at the roots]
If $t$ is a root of $H$, then
\[
|H'(t)|>\kappa_n\gamma_n,\qquad \gamma_n>\sqrt m.
\]
Consequently,
\[
|R|>\frac{a^{e-k}}{V_A}(\kappa_n\gamma_n)^k.
\]
\end{lemma}

\begin{proof}
Define
\[
E_n(x)=(1-x^2)P_n'(x)^2+n(n+1)P_n(x)^2.
\]
Differentiation and the differential equation~(1.2) give
\begin{equation}
\begin{aligned}
E_n'(x)
&=-2xP_n'(x)^2+2P_n'(x)\bigl((1-x^2)P_n''(x)+n(n+1)P_n(x)\bigr)\\
&=2xP_n'(x)^2.
\end{aligned}
\label{eq:5.6}
\end{equation}
Since $P_n'$ is a nonzero polynomial, $E_n(x)>E_n(0)=\gamma_n^2$ for every $x>0$. Let $t\in(0,1)$ be a root of $H$ and take $x=\sqrt t$. The identity $\kappa_nP_n(x)=x^\varepsilon H(x^2)$ gives $P_n(x)=0$, hence
\[
|P_n'(\sqrt t)|>\frac{\gamma_n}{\sqrt{1-t}}.
\]
Differentiating the normalization identity at the root gives
$\kappa_nP_n'(\sqrt t)=2t^{(\varepsilon+1)/2}H'(t)$. Therefore
\begin{equation}
|H'(t)|>
\frac{\kappa_n\gamma_n}{2t^{(\varepsilon+1)/2}\sqrt{1-t}}
\ge\kappa_n\gamma_n,
\label{eq:5.7}
\end{equation}
where we used
$2t^{(\varepsilon+1)/2}\sqrt{1-t}\le2\sqrt{t(1-t)}\le1$.
The claimed bound for $|R|$ now follows from Lemma~5.3.

To estimate $\gamma_n$, put $b_m=4^{-m}\binom{2m}{m}$. The coefficient formula~(2.2) gives
\[
\gamma_{2m}^2=2m(2m+1)b_m^2,\qquad
\gamma_{2m+1}^2=(2m+1)^2b_m^2.
\]
We have $b_1=1/2$ and
\[
\frac{b_{m+1}}{b_m}=\frac{2m+1}{2m+2},\qquad
\left(\frac{2m+1}{2m+2}\right)^2-\frac m{m+1}
=\frac1{4(m+1)^2}>0.
\]
Induction yields $b_m\ge1/(2\sqrt m)$. Thus
\[
\gamma_{2m}^2\ge m+\frac12>m,\qquad
\gamma_{2m+1}^2\ge\frac{(2m+1)^2}{4m}>m,
\]
proving the remaining assertion.
\end{proof}

\begin{lemma}[The product of root differences in an interval]
If $t_1,\ldots,t_k\in[0,1]$, then
\begin{equation}
\prod_{1\le i<j\le k}(t_i-t_j)^2
\le\frac{k^k}{2^{2(k-1)^2}}.
\label{eq:5.8}
\end{equation}
\end{lemma}

\begin{proof}
Let $T_j$ be the Chebyshev polynomial of the first kind, defined by $T_j(\cos\theta)=\cos(j\theta)$. Then $|T_j(2t-1)|\le1$ for $0\le t\le1$. From $T_0=1$, $T_1=x$, and the recurrence $T_{j+1}=2xT_j-T_{j-1}$, induction shows that the leading coefficient of $T_j$ is $2^{j-1}$ for $j\ge1$. Thus the leading coefficient of $T_j(2t-1)$ is $2^{2j-1}$.

Consider the matrix $D=(T_{j-1}(2t_i-1))_{1\le i,j\le k}$. The Euclidean norm of each column is at most $\sqrt k$, so Theorem~1.6 gives
\[
|\det D|^2\le k^k.
\]
Using the leading coefficients above and the Vandermonde determinant formula, we find
\[
\det D
=\left(\prod_{j=1}^{k-1}2^{2j-1}\right)
\prod_{1\le i<j\le k}(t_j-t_i)
=2^{(k-1)^2}\prod_{1\le i<j\le k}(t_j-t_i).
\]
Taking the squared absolute value proves the desired inequality; for $k=1$, both sides are $1$.
\end{proof}

\begin{proof}[\textbf{Proof of Theorem 5.1}]
Apply Proposition~5.2 and set $d=2k-s_2(k)$, so that
\[
A(1)=2^d,\qquad a>A(1),\qquad v_2(R)\le ed.
\]
Since $e\ge k$, Lemma~5.4 gives
\begin{equation}
\begin{aligned}
\mathcal R
&>\frac{a^{e-k}}{2^{ed}V_A}(\kappa_n\gamma_n)^k\\
&\ge\frac1{V_A}\left(\frac{\kappa_n\gamma_n}{2^d}\right)^k
=\frac1{V_A}\left(2^{2e-s_2(e)}\gamma_n\right)^k.
\end{aligned}
\label{eq:5.9}
\end{equation}
The last equality uses $s_2(m)=s_2(k)+s_2(e)$ and
$\kappa_n/2^d=2^{2e-s_2(e)}$.
Applying Lemma~5.5, taking logarithms, and dividing by $k$, we obtain
\begin{equation}
\begin{aligned}
\frac{\log\mathcal R}{k}
&>(2e-s_2(e))\log2+\log\gamma_n
  +\frac{2(k-1)^2}{k}\log2-\log k\\
&=m\log4-\log\bigl(2^{s_2(e)+4}k\bigr)
  +\log\gamma_n+\frac{2\log2}{k}.
\end{aligned}
\label{eq:5.10}
\end{equation}
An integer with $s$ nonzero binary digits is at least $2^s-1$, so
$2^{s_2(e)}\le e+1$. Since $k+(e+1)=m+1$,
\[
2^{s_2(e)+4}k\le16k(e+1)\le4(m+1)^2.
\]
Substituting this and $\gamma_n>\sqrt m$ from Lemma~5.4 into~(5.10), and discarding the positive term $2\log2/k$, proves~(5.1).
\end{proof}

\section{Comparison of the two bounds and proof of the main theorem}

The preceding two sections establish upper and lower bounds for the odd part of the same resultant. We first show that these bounds are incompatible for every $n\ge256$, then use Theorem~1.8 and explicit integer comparisons to handle the remaining degrees. We continue to write $n=2m+\varepsilon$, $m=\lfloor n/2\rfloor$, and $M_n=\kappa_nP_n/x^\varepsilon$. Under the assumed factorization, $k=\deg A\le e=\deg B=m-k$ and $\mathcal R=|\operatorname{Res}_t(A,B)|/2^{v_2(R)}$.

\begin{theorem}[Irreducibility in sufficiently large degrees]
For every integer $n\ge256$, the polynomial $M_n$ is irreducible in $\mathbb Q[x]$.
\end{theorem}

\begin{proof}
Suppose that $n\ge256$ and, to the contrary, that $M_n$ is reducible. Then $m\ge128$. By~(4.1), (5.1), and $2n-1\le4m+1$, it is enough to prove
\begin{equation}
m\log(4/3)>
\sqrt{4m+1}\log3+\log\frac{4(m+1)^2}{\sqrt m}.
\label{eq:6.1}
\end{equation}
For real $x\ge128$, define
\[
f(x)=x\log(4/3)-\sqrt{4x+1}\log3-\log4-2\log(x+1)+\tfrac12\log x.
\]
The elementary inequalities
\[
\log(4/3)>\frac27,\qquad \log3<\frac{11}{10},\qquad \log2<\frac7{10}
\]
follow, respectively, from
$\log((1+u)/(1-u))>2u$ at $u=1/7$, and from the positive terms of the exponential series. Since
$\sqrt{513}<23$ and $4\cdot129^2/\sqrt{128}<2^{13}$,
\[
f(128)>
\frac{256}{7}-23\cdot\frac{11}{10}-13\cdot\frac7{10}
=\frac{76}{35}>0.
\]
Moreover, for every $x\ge128$,
\begin{equation}
\begin{aligned}
f'(x)
&=\log(4/3)-\frac{2\log3}{\sqrt{4x+1}}-\frac2{x+1}+\frac1{2x}\\
&>\frac27-\frac1{10}-\frac2{129}>0,
\end{aligned}
\label{eq:6.2}
\end{equation}
where $\sqrt{4x+1}>22$ was used. Hence $f(m)>0$, proving~(6.1). The lower bound~(5.1) therefore exceeds the upper bound~(4.1), a contradiction.
\end{proof}

\begin{lemma}[The remaining degrees]
For every integer $2\le n\le255$, the polynomial $M_n$ is irreducible in $\mathbb Q[x]$.
\end{lemma}

\begin{proof}
Theorem~1.8(1) covers $2\le n\le121$. Among $122\le n\le255$, Theorem~1.8(2) covers every degree except
\begin{equation}
\begin{split}
\mathcal S=\{&122,143,144,185,186,203,204,205,206,\\
             &215,216,217,218,245,246\}.
\end{split}
\label{eq:6.3}
\end{equation}
Indeed, the primes from $127$ to $257$ are
\[
\begin{gathered}
127,131,137,139,149,151,157,163,167,173,\\
179,181,191,193,197,199,211,223,227,229,\\
233,239,241,251,257;
\end{gathered}
\]
their intervals $[p-4,p+3]$ leave exactly the degrees in~(6.3).

Fix $n\in\mathcal S$ and suppose that $M_n$ is reducible. All these degrees exceed $64$, so both~(4.3) and~(5.1) apply. For the corresponding values $61\le m\le123$, the function $4(x+1)^2/\sqrt x$ is increasing for $x\ge1$, and hence
\begin{equation}
\frac{4(m+1)^2}{\sqrt m}
\le\frac{4\cdot128^2}{\sqrt{127}}<2^{13}.
\label{eq:6.4}
\end{equation}
Direct multiplication of the prime powers in~(4.2) gives the integers $r_n$ in the following table, with
\[
D_n<2^{r_n}\le2^{2m-13}.
\]
\begin{center}
\begin{tabular}{rrrr}
\toprule
$n$ & $m$ & $r_n$ & $2m-13$\\
\midrule
122&61&97&109\\
143&71&109&129\\
144&72&109&131\\
185&92&146&171\\
186&93&146&173\\
203&101&159&189\\
204&102&159&191\\
205&102&159&191\\
206&103&166&193\\
215&107&172&201\\
216&108&172&203\\
217&108&172&203\\
218&109&179&205\\
245&122&186&231\\
246&123&186&233\\
\bottomrule
\end{tabular}
\end{center}
These comparisons involve only integers: each row asserts that the product in~(4.2) is smaller than the indicated power of $2$. By~(5.1) and~(6.4),
\[
\frac{\log\mathcal R}{k}>(2m-13)\log2>\log D_n,
\]
contrary to~(4.3). This handles every degree in $\mathcal S$ and completes the proof.
\end{proof}

\begin{proof}[\textbf{Proof of the main theorem (Theorem 1.1)}]
Theorem~6.1 covers all original degrees $n\ge256$, and Lemma~6.2 covers $2\le n\le255$. Finally, $M_n=\kappa_nP_n(x)/x^\varepsilon$ and $\kappa_n$ is a nonzero rational number, so multiplication by $\kappa_n$ does not affect irreducibility in $\mathbb Q[x]$. Taking $n=2j$ and $n=2j+1$, respectively, proves that $P_{2j}(x)$ and $P_{2j+1}(x)/x$ are irreducible in $\mathbb Q[x]$ for every $j\ge1$.
\end{proof}

\end{document}